\documentclass{aptpub}

\authornames{Frank Ball, Tom Britton and Peter Neal} 
\shorttitle{On expected durations of birth-death processes} 

\usepackage[dvips]{psfrag, graphicx, epsfig}
\usepackage{natbib}

\newcommand{\convd}{\stackrel{D}{\longrightarrow}}

\newcommand{\me}{\mathrm{e}}
\newcommand{\mE}{\mathrm{E}}
\newcommand{\mP}{\mathrm{P}}
\newcommand{\mc}{\mathrm{C}}

\newcommand{\bmpi}{\mbox{\boldmath{$\pi$}}}

\newcommand{\bmphi}{\mbox{\boldmath{$\phi$}}}

\numberwithin{equation}{section}

\begin{document}

\title{On expected durations of birth-death processes, with applications to branching
processes and SIS epidemics} 

\authorone[University of Nottingham]{Frank Ball} \authortwo[Stockholm University]{Tom Britton} \authorthree[Lancaster University]{Peter Neal} 

\addressone{The Mathematical Sciences Building,
University Park, Nottingham, NG7 2RD, UK}
\addresstwo{Department of Mathematics, Stockholm University,
SE-106 91 Stockholm, Sweden.}
\addressthree{Department of Mathematics and Statistics,
Fylde College, Lancaster University, LA1 4YF,
United Kingdom} 

\begin{abstract}
We study continuous-time birth-death type processes, where
individuals have independent and identically distributed lifetimes,
according to a random variable $Q$, with $\mE[Q]=1$, and where the
birth rate if the population is currently in state (has size) $n$ is
$\alpha(n)$. We focus on two important examples, namely
$\alpha(n)=\lambda n$ being a branching process, and
$\alpha(n)=\lambda n(N-n)/N$ which corresponds to an SIS
(susceptible $\to$ infective $\to$ susceptible) epidemic model in a
homogeneously mixing community of fixed size $N$. The processes are
assumed to start with a single individual, {\it i.e.}~in state 1.
Let $T$, $A_n$, $C$ and $S$ denote the (random) time to extinction,
the total time spent in state $n$, the total number of individuals
ever alive and the sum of the lifetimes of all individuals in the
birth-death process, respectively. The main results of the paper
give expressions for the expectation of all these quantities, and
shows that these expectations are insensitive to the distribution of
$Q$. We also derive an asymptotic expression for the expected time to
extinction of the SIS epidemic, but now starting at the endemic
state, which is \emph{not} independent of the distribution of $Q$.
The results are also applied to the household SIS epidemic, showing
that its threshold parameter $R_*$ is insensitive to the
distribution of $Q$, contrary to the household SIR (susceptible
$\to$ infective $\to$ recovered) epidemic, for which $R_*$ does
depend on $Q$.
\end{abstract}

\keywords{Birth-death process; branching processes; SIS
epidemics; insensitivity results.} 

\ams{60J80}{60G10;92D30} 

\section{Introduction} \label{S:intro}

A key question for population processes of a birth-death type, for
example, branching processes and epidemic processes (with infection and
recovery corresponding to birth and death, respectively), is what
effect does the lifetime distribution have on key quantities of
scientific interest? For example, consider a single-type branching
process, where individuals have independent and identically
distributed ({\it iid}) lifetimes according to a random variable $Q$ having an arbitrary, but
specified, distribution and, whilst alive, give birth at the
points of a homogeneous Poisson point process with rate $\lambda$.
The basic reproduction number, $R_0 = \lambda \mE [Q]$, the mean number of
offspring produced by an individual during its lifetime, only
depends upon $Q$ through its mean  $\mE[Q]$. The mean total size of a
subcritical branching process $(R_0 < 1)$ with one ancestor is $1/(1-R_0)$, which is again
independent of the distribution of $Q$. However, other quantities of
interest such as the probability of extinction and the Malthusian parameter of the branching process depend upon the distributional form
of $Q$. Thus, in the language of stochastic networks, $R_0$ can be
viewed as an {\it insensitivity} result in that it depends on $Q$
only through its mean, see, for example, \cite{Z07}. Without loss of
generality throughout the paper we assume that $\mE[Q]=1$.

Insensitivity results for stochastic networks are well known, see
for example, \cite{Sev57}, \cite{Whittle85} and \cite{Z07}. In
particular, in \cite{Z07}, Theorem 1, it is shown that for a wide
class of queueing networks, where arrivals (births) into the system
are Poissonian with rate depending upon the total number of
individuals in the system and each arrival has an {\it iid}
workload, the stationary distribution of the total number of
individuals in the system is insensitive to the distribution of $Q$.
It then follows automatically that, for example, the mean duration
of a busy period of the network (at least one individual in the
system) is insensitive to the distribution of $Q$.

Given the similarities between queueing networks and birth-death
type models, arrivals equating to births and workload equating to
lifetime, we seek in this paper to explore insensitivity results for
birth-death type processes with particular emphasis upon branching
processes and SIS (susceptible $\to$ infective $\to$ susceptible) epidemic models. In many cases,  \cite{Z07},
Theorem 1, cannot be applied directly to birth-death processes, as
many birth-death processes do not exhibit stationary behaviour. For
example, a branching process will either go extinct or grow
exponentially. However, we can exploit \cite{Z07}, Theorem 1, for
birth-death type processes whose mean time to extinction is finite
by introducing a regeneration step ({\it cf.}~\cite{HSCC99}) whenever the
population goes extinct. That is, whenever the population goes
extinct, it spends an exponential length of time in state 0 (no
individuals) before a new individual is introduced into the
population (regeneration). The birth-death type process with
regenerations then fits into the framework of \cite{Z07}, provided
that the birth rate is Poissonian and depends upon the population
only through its size.
Insensitivity results are then easy to obtain for the regenerative
process, and also for the original birth-death type process. In
particular, we obtain that the mean duration of the birth-death type
process is insensitive to the distribution of $Q$.

The paper is structured as follows. In Section \ref{S:generic}, we
formally introduce the generic birth-death type process with
arbitrary birth rate $\alpha (n)$, where $n$ denotes population size, and introduce regeneration. We
identify key insensitivity results for birth-death type processes,
namely, that the mean duration, the mean time with $n$ individuals alive $(n=1,2,\ldots)$ and the mean total number of
individuals ever alive in the process are insensitive to the
distribution of $Q$. In Section \ref{S:special}, we focus on three
special cases of the birth-death type process, namely, branching
processes with constant birth rate, and homogeneously mixing and
household SIS epidemic models. In Section \ref{S:BP}, we prove a
conjecture of \cite{N14}, that for a subcritical branching process,
the mean time with $n$ $(n=1,2,\ldots)$ individuals alive is
insensitive to $Q$ and, using \cite{Lambert11}, Lemma 3.1, give a
corresponding insensitivity result for critical and supercritical
branching processes. In Section \ref{S:SIS}, we apply the
insensitivity results to homogeneously mixing SIS epidemics and obtain a
simple approximation for the mean duration of the epidemic starting
from a single infective. Moreover, we show that for a supercritical
epidemic ($R_0>1$), the mean duration of the epidemic starting from the
quasi-endemic equilibrium does depend upon the distribution of $Q$ and we give a simple
asymptotic expression for this quantity. Finally, in Section
\ref{S:house} we exploit the results obtained for the homogeneously
mixing SIS epidemic to show that both the threshold parameter $R_\ast$
and the quasi-endemic equilibrium of the household SIS epidemic are
insensitive to the distribution of $Q$. These are interesting findings, as in
the household SIR (susceptible $\to$ infective $\to$ recovered) epidemic both
$R_\ast$ and the fraction of the population ultimately recovered if the epidemic takes off do depend upon the distribution of $Q$.

\section{Generic model} \label{S:generic}

The generic birth-death type process is defined as follows. The
process is initiated at time $t=0$ with one individual. All
individuals, including the initial individual, have {\it iid} lifetimes according to an arbitrary, but
specified, positive random variable $Q$ with finite mean. At the end
of its lifetime an individual dies and is removed from the
population. New individuals are born and enter the population at the
points of an independent inhomogeneous Poisson point process with
rate $\alpha (n) \geq 0$, where $n$ denotes the total number of
individuals in the population. Without loss of generality, we assume
that $\mE[Q]=1$, since otherwise we can simply rescale time by
dividing $Q$ and multiplying $\alpha (n)$ by $\mE[Q]$. The special
cases of a branching process with individuals giving birth at the
points of a homogeneous Poisson point process with rate $\lambda$
and the homogeneously mixing SIS epidemic (see, for
example,~\cite{Kryscio}) in a population of size $N$ with infection
rate $\lambda$, correspond to $\alpha (n) = n \lambda$ and $\alpha
(n) = n \lambda (N-n)/N$, respectively.

The birth-death type process is similar to the single-class networks
studied in \cite{Z07}, Section 2. In \cite{Z07}, it is assumed that
new individuals enter the system (births) at the points of a Poisson process with
state-dependent rate $\alpha (n)$, where $n$ is the total number of
individuals currently in the system. Individuals have independent
and identically distributed workloads, according to a random
variable $Q$ with $\mE[Q]=1$. While there are $n$ individuals in the
system, the total workload is reduced at rate $\beta (n) \geq 0$,
with $\beta (n) > 0$ if and only if $n >0$. In a biological setting,
where the workload $Q$ associated with an individual is its
lifetime, it only makes sense to take $\beta (n)=n$, so each
individual's remaining lifetime decreases at a constant rate 1. In
queueing terminology this corresponds to an infinite server queue.

In \cite{Z07}, Theorem 1, it is shown that if the proper distribution $\bmpi = (\pi(0),
\pi(1), \ldots)$ satisfies the detailed balance equations
\begin{eqnarray} \label{eq:zachary:1}
\pi (n+1) \beta (n+1) = \pi (n) \alpha (n) \hspace{0.5cm} n=0,1,
\ldots,, \end{eqnarray} and
\begin{eqnarray} \label{eq:zachary:2}
\sum_{n=0}^\infty \pi (n) \alpha (n)<\infty,
\end{eqnarray}
then $\bmpi$ is the stationary distribution of the size of the
system, irrespective of the distribution of $Q$. That is, if $Z_t$
denotes the number of individuals in the system at time $t$, then
for $n=0,1,\ldots$
\begin{eqnarray} \label{eq:zachary:0}
\mP (Z_t = n) \rightarrow \pi (n) \hspace{0.5cm} \mbox{as } t
\rightarrow \infty.
\end{eqnarray}
Note that the total number of individuals in the system is not a
Markov process, unless $Q$ has an exponential distribution.

For many biological systems, \cite{Z07} does not apply since often
$\alpha (0)=0$. That is, the population can go extinct and then
remains extinct forever. This is the case for branching processes and the
homogeneously mixing SIS epidemic model.  Moreover, for a branching
process if $\lambda \leq1$ (subcritical/critical), the branching
process goes extinct with probability 1, whereas for a
supercritical branching process $\lambda >1$, the branching process
either goes extinct or grows unboundedly. In either case, a
stationary distribution for the total number of individuals alive
does not exist. The solution to make \cite{Z07}, Theorem 1, relevant
to birth-death type processes is to follow \cite{HSCC99} and
consider a birth-death type process with regeneration. We introduce
regeneration by setting $\alpha (0)=1$, leaving all other
transition rates unchanged. This corresponds to the process, if it goes
extinct, spending an exponentially distributed time, having mean 1, with no
individual before a new individual enters the population leading to
the process restarting (regeneration).

The key questions are, how to analyse the regenerative process and
what does it tell us about the original birth-death type process?
Firstly, if the regenerative process satisfies the detailed balance equation
\eqref{eq:zachary:1} and is non-explosive (satisfies condition
\eqref{eq:zachary:2}), then it has a stationary
distribution $\bmpi$, which is insensitive to the distribution of
$Q$. Secondly, the behaviour of the process between regenerations
are independent and identically distributed copies of the original
birth-death process. Therefore this gives us a way to explore
characteristics of the original process, and of identifying quantities
which are insensitive to the distribution of $Q$ and also, later in
Section \ref{S:SIS}, some which do depend upon the distribution of
$Q$. Recall that in the processes we study $\beta(n)=n$.  Then~\eqref{eq:zachary:1} implies that
\begin{eqnarray} \label{eq:stat:1}
 \pi (n) \alpha (n) &=& (n+1) \pi (n+1)
\hspace{0.5cm} \mbox{for } \; n=0,1,\ldots. \end{eqnarray} Therefore,
it follows that
\begin{eqnarray} \label{eq:stat:3} \pi (n) &=& \pi (0)
\prod_{i=1}^{n-1} \frac{\alpha (i)}{i+1},\hspace{0.5cm} n=1,2,
\ldots,
\end{eqnarray}
where the product is $1$ when $n=1$.
It then follows from \eqref{eq:stat:1} and \eqref{eq:stat:3} that
\begin{eqnarray} \label{eq:stat:4} \pi (0) &=& \left\{ 1 +
\sum_{n=1}^\infty \prod_{i=0}^{n-1} \frac{\alpha (i)}{i+1}
\right\}^{-1}.
\end{eqnarray}
Note that $\bmpi$ being a proper distribution implicitly implies that $\pi(0)>0$ or, equivalently, that the sum in~\eqref{eq:stat:4} is finite, and hence that the process is positive recurrent.
Thus the regenerative
process is not suitable for critical or supercritical branching
processes. We discuss this in more detail in Section
\ref{S:BP} below.

We complete this section by identifying a number of key quantities
whose means are insensitive to the distribution of $Q$ and are
summarised in Theorem \ref{thm1}.
\begin{thm} \label{thm1}
Let $Y_t$ denote the total number of individuals in the birth-death
process at time $t$. Let $T = \int_0^\infty 1_{\{Y_t
> 0\}} \, dt$ denote the duration of the birth-death process and,
for $n=1,2,\ldots$, let $A_n= \int_0^\infty 1_{\{Y_t =n\}} \, dt$
denote the total time the birth-death process spends with $n$
individuals alive. Then
\begin{eqnarray} \label{eq:stat:5} \mE[T] =
\sum_{n=1}^\infty \prod_{i=0}^{n-1} \frac{\alpha (i)}{i+1} ,
\end{eqnarray}
and for $n=1,2,\ldots$,
\begin{eqnarray} \label{eq:stat:6} \mE [A_n] &=&
\prod_{i=0}^{n-1} \frac{\alpha (i)}{i+1}.
\end{eqnarray}
Finally, let $C$ and $S$ denote the total number of individuals ever
alive in the population and the sum of the lifetimes of all
individuals ever alive in the birth-death process, respectively.
Then $\mE[C]= \mE[S]$ with
\begin{eqnarray} \label{eq:stat:9}
\mE[S] &=& \sum_{k=1}^\infty k\prod_{i=1}^{k-1} \frac{\alpha
(i)}{i+1}.
\end{eqnarray}
\end{thm}
\begin{proof}
An immediate consequence of the above construction is that the mean
time between regenerations is $1/\pi(0)$, irrespective of the
distribution of $Q$. On average one unit of time is spent with no
individual in the population, so (see~\cite{BM04} for a formal
justification)
\begin{eqnarray} \label{eq:stat:5a} \mE[T] =\frac{1}{\pi (0)} -1 = \left\{
\sum_{n=1}^\infty \prod_{i=0}^{n-1} \frac{\alpha (i)}{i+1} \right\},
\end{eqnarray} as required.
Moreover, for $n=1,2,\ldots$,
\begin{eqnarray} \label{eq:stat:6a} \mE [A_n] &=& \pi (n) \frac{1}{\pi(0)}
=\prod_{i=0}^{n-1} \frac{\alpha (i)}{i+1},
\end{eqnarray}
using~\eqref{eq:stat:3}.

Using Fubini's theorem,
\begin{eqnarray} \label{eq:stat:7}
\mE [S]&=& \mE \left[ \int_0^\infty Y_t \, dt \right] \nonumber \\
&=& \mE \left[ \int_0^\infty \sum_{k=1}^\infty k 1_{\{Y_t = k \}} \, dt \right] \nonumber \\
&=& \sum_{k=1}^\infty k \mE \left[\int_0^\infty 1_{\{Y_t = k \}} \, dt \right] \nonumber \\
&=& \sum_{k=1}^\infty k \mE[A_k] = \sum_{k=1}^\infty
k\prod_{i=1}^{k-1} \frac{\alpha (i)}{i+1}.
\end{eqnarray}
Note that, after using~\eqref{eq:zachary:1} with $\beta(n)=n$,~\eqref{eq:zachary:2} ensures that $\mE[S]$ is finite.

Given $A_k$, the mean number of births whilst the process is in
state $k$ is $ \alpha (k) A_k$. Therefore, including the initial
ancestor and noting from~\eqref{eq:stat:6a} that $\mE[A_1] = 1$, we have that
\begin{eqnarray}\label{eq:stat:8}
\mE[C] &=& \mE \left[ 1 + \sum_{k=1}^\infty \alpha (k) A_k \right] \nonumber \\
&=& 1+ \sum_{k=1}^\infty \alpha (k) \frac{ \pi (k)}{ \pi(0)} \nonumber \\
&=& \mE[A_1]+ \sum_{k=1}^\infty (k+1) \frac{ \pi (k+1)}{ \pi(0)} \nonumber \\
&=& \mE[A_1]+ \sum_{k=2}^\infty k \mE[A_k]\label{eq:stat8a}\\
&=& \mE[S].
\end{eqnarray}
\end{proof}

\section{Special cases} \label{S:special}

In this section we apply the results obtained in Section
\ref{S:generic} to three special cases, namely branching processes,
homogeneously mixing SIS epidemics and household SIS epidemics,
yielding fresh insight into these models.

\subsection{Branching process} \label{S:BP}

As mentioned above, we consider branching processes where
individuals have {\it iid} lifetimes  according to $Q$ (with
$E[Q]=1$) and whilst alive give  birth at the points of a
homogeneous Poisson point process with rate $\lambda$. Therefore we
have that $\alpha (n) = n \lambda$. The key result is Lemma
\ref{lem:BP} which is a generalization of \cite{N14}, Conjecture
2.1.
\begin{lem} \label{lem:BP}
For $n=1, 2,\dots$,
\begin{equation}
\mE[A_n] =\frac{\lambda^ {n-1}}{n(\max \{1, \lambda
\})^n}.\label{main-result}
\end{equation}
\end{lem}
The Lemma is proved in \eqref{eq:bp:3} and \eqref{eq:bp:6} below.

For $\lambda <1$, the branching process is subcritical and the
results of Section \ref{S:generic} hold. The form of $\alpha (n)$
allows explicit expressions to be obtained. Firstly, it follows from
\eqref{eq:stat:3} that
\begin{eqnarray} \label{eq:bp:1} \pi (n) &=&
\pi (0) \prod_{i=1}^{n-1} \frac{i \lambda}{i+1}
= \pi (0) \frac{\lambda^{n-1}}{n},\hspace{0.5cm} n=1,2,
\ldots.
\end{eqnarray}
It then follows immediately that
\begin{eqnarray}  \pi (0) &=& \left\{ 1 +
\sum_{n=1}^\infty \frac{\lambda^{n-1}}{n} \right\}^{-1} \label{eq:bp:2a}
\\
&=&  \left\{ 1 -\log(1-\lambda)/\lambda  \right\}^{-1}.\label{eq:bp:2}
\end{eqnarray}
(Note that the sum in~\eqref{eq:bp:2a} diverges if $\lambda \ge 1$.)
Therefore, it follows from \eqref{eq:bp:2} and \eqref{eq:stat:5} that
the mean duration of the branching process is
\begin{equation}
\label{meansubcritext}
\mE[T] =-\log (1-\lambda)/\lambda,
\end{equation}
a result which is well known for the linear
birth-death process with $Q \sim {\rm Exp} (1)$. Also \cite{N14},
Conjecture 2.1, is proved in that, for $n=1,2,\ldots$,
\begin{eqnarray} \label{eq:bp:3} \mE[A_n] &=& \frac{\lambda^{n-1}}{n}.
\end{eqnarray}
Finally, we obtain the classical result that the mean total number
of individuals ever alive in the branching process is
\begin{eqnarray}  \label{eq:bp:4} \mE[C]&=& \sum_{k=1}^\infty k \mE[A_k]= \sum_{k=1}^\infty \lambda^{k-1} = \frac{1}{1-\lambda}.
\end{eqnarray}

The approach taken in Section \ref{S:generic} is valid only in the
above scenario. A key component of the proofs is that the birth rate
depends only upon the total number of individuals in the population.
Therefore the above insensitivity results do not hold for more
general reproductive life histories. However, progress can be made
in extending \cite{N14}, Conjecture 2.1, and \eqref{eq:bp:3} to the
critical and supercritical cases ($\lambda \geq 1$) by using
\cite{Lambert11}, Lemma 3.1. (Note that the mean duration and mean
total number of individuals ever alive in the branching process is infinite  in these cases.) Specifically, \cite{Lambert11}, Lemma
3.1, shows that, for $n=1,2,\ldots$ and $t \geq 0$,
\begin{eqnarray} \label{eq:bp:5} \mP (Y_t = n) = \left(1 -
\frac{1}{W(t)} \right)^{n-1} \frac{W^\prime (t)}{\lambda W (t)^2},
\end{eqnarray} where $W(t)$ solves \cite{Lambert11}, equation (6). For most
choices of $Q$, it is not possible to get an explicit expression for
$W(t)$ for all $t \geq 0$. However,  for any $Q$ ($\mE[Q]=1$) $W(0)=1$ and for $\lambda \geq 1$,
$W(\infty)=\infty$. Therefore
\begin{eqnarray} \label{eq:bp:6} \mE [A_n] &=& \int_0^\infty \mP (Y_t =
n) \, dt \nonumber \\
&=& \frac{1}{\lambda} \int_0^\infty \left(1 - \frac{1}{W(t)}
\right)^{n-1} \frac{W^\prime (t)}{W (t)^2} \, dt \nonumber  \\
&=& \frac{1}{\lambda} \left[ \frac{1}{n} \left( 1- \frac{1}{W(t)}
\right)^n \right]_0^\infty = \frac{1}{\lambda n}. \end{eqnarray} It
should be noted that in the subcritical case ($\lambda <1$)
$W(\infty)=1/(1-\lambda)$ and \eqref{eq:bp:6} can be used to obtain
\eqref{eq:bp:3} directly.

\subsection{SIS epidemic} \label{S:SIS}

\subsubsection{Mean duration with one initial infective} \label{S:SIS1}

For a homogeneously mixing SIS epidemic in a population of size $N$,
the rate at which new infections (births) occur, given that there
are currently $n$ infectives is $\alpha (n) = \lambda n (N-n)/N$,
where $\lambda$ is the rate at which a typical infective makes
infectious contacts. Infectious contacts are assumed to be with
individuals chosen independently and uniformly  at random from the
entire population.  Thus, if there are $y$ infectives, and hence
$N-y$ susceptibles, the probability that an infectious contact is
with a suscpetible, and hence results in a new infective, is
$(N-y)/N$.  The infectious periods of infectives are independently
and identically distributed according to $Q$ and an infective
becomes susceptible again as soon as its infectious period ends.
Using \eqref{eq:stat:3}, it is straightforward to show that, for
$n=1,2,\ldots, N$,
\begin{eqnarray} \label{eq:SIS:1} \pi^{(N)} (n) = \frac{\pi^{(N)} (0)}{n} \frac{(N-1)!}{(N-n)!} \left(
\frac{\lambda}{N}\right)^{n-1},
\end{eqnarray} which has previously been obtained for the Markov
case ($Q \sim {\rm Exp} (1)$) by~\cite{HSCC99}.
Consequently, the mean duration of the epidemic starting from a single infective is
\begin{eqnarray} \label{eq:SIS:2} \mE\left[ T^{(N)}\right] &=& \sum_{n=1}^N
\frac{(N-1)!}{n (N-n)!} \left( \frac{\lambda}{N} \right)^{n-1}.
\end{eqnarray} Note that this follows directly from~\eqref{eq:stat:5}, on recalling that there $\alpha(0)=1$.  We index statistics of interest by the total population size $N$ to highlight the role played by $N$ in the analysis below.

It is interesting to investigate the behaviour of $\mE\left[ T^{(N)}\right]$ for
large $N$. We summarise the results in Lemma \ref{lem:sis}.
\begin{lem} \label{lem:sis}
For the subcritical case,  $\lambda <1$,
\begin{eqnarray} \label{eq:ls1} \mE\left[T^{(N)}\right]\rightarrow - \frac{\log (1-\lambda)}{\lambda}
\hspace{0.5cm} \mbox{as } N \rightarrow \infty. \end{eqnarray}

For the critical case, $\lambda =1$,
\begin{eqnarray} \label{eq:ls2} \mE\left[T^{(N)}\right] \sim \frac{1}{2} \log N,
\end{eqnarray} where $a^{(N)} \sim b^{(N)}$ denotes that $\lim_{N
\rightarrow \infty} a^{(N)}/b^{(N)} =1$.

For the supercritical case, $\lambda >1$,
\begin{eqnarray} \label{eq:ls3} \mE\left[T^{(N)}\right] \sim \frac{\sqrt{2 \pi}}{\lambda -1} \frac{\exp (\{ \log \lambda -1 +1/\lambda \}
N)}{\sqrt{N}}.
\end{eqnarray}
\end{lem}

We give a heuristic proof of Lemma~\ref{lem:sis} here, deferring a formal proof to Appendix A.

For $\lambda <1$, we have that $\pi^{(N)} (n) \approx \pi^{(N)} (1)
\lambda^{n-1}/n$, a branching process approximation ({\it
c.f.}~\cite{Whittle}, \cite{BallDonnelly}) and $\mE\left[ T^{(N)}\right] \approx
-\log (1- \lambda)/\lambda$ ({\it c.f.}~Section \ref{S:BP}).

For $\lambda=1$, it is fruitful to recall the birthday problem.  Balls are drawn uniformly at random with replacement from an urn containing $N$ balls numbered $1,2,\dots,N$.  Let $M^{(N)}$ be the number of draws required until a ball is drawn that had previously been drawn.  Then~\eqref{eq:SIS:2} yields, with products being one if vacuous,
\begin{eqnarray}
\mE\left[T^{(N)}\right]&=&\sum_{n=1}^N \frac{1}{n} \prod_{i=1}^{n-1} \left(1 - \frac{i}{N} \right)\nonumber\\
&=&\sum_{n=1}^N \frac{1}{n} \mP\left(M^{(N)}>n\right)\nonumber\\
&=&\mE\left[\sum_{n=1}^{M^{(N)}} \frac{1}{n}\right].
\end{eqnarray}
Now, for example,~\cite{Aldous85}, page 96,
\begin{equation}
\label{birthday}
N^{-\frac{1}{2}} M^{(N)} \overset{\mathrm{D}}{\longrightarrow} M \quad\mbox{as}\quad N \to \infty,
\end{equation}
where $\overset{\mathrm{D}}{\longrightarrow}$ denotes convergence in distribution and $M$ is a random variable having probability density function $f_M(x)=x\exp\left(-\frac{x^2}{2}\right)$ $(x > 0)$.  Thus, when $N$ is large, $M^{(N)}$ is concentrated on $n$-values of the form $x N^{\frac{1}{2}}$, for which
\[
\sum_{i=1}^n \frac{1}{i} \approx \log\left( x N^{\frac{1}{2}}\right)=\log x + \frac{1}{2} \log N
\]
and~\eqref{eq:ls2} follows.

For $\lambda >1$, rearranging \eqref{eq:SIS:2} we get
\begin{equation} \label{eq:SIS:3}
\mE\left[ T^{(N)}\right]=(N-1)!(\lambda/N)^{N-1}\sum_{j=0}^{N-1}
\frac{(N/\lambda)^{j}}{(N-j)j!}.
\end{equation}
Let us first consider the sum in~\eqref{eq:SIS:3}. If we multiply by
$\me^{-N/\lambda}$ we are adding Poisson($N/\lambda$) probabilities
multiplied by $1/(N-j)$. When $N$ is large the probability mass is
concentrated on $j$-values near the mean, $j\approx N/\lambda$, by
the law of large numbers. Owing to the assumption $\lambda>1$, these
terms are smaller than $N$ and hence contained in the finite sum.
From this it follows that the sum is asymptotically equivalent to
\begin{eqnarray} \label{eq:dur:1}
\sum_{j=0}^{N-1} \frac{(N/\lambda)^{j}}{(N-j)j!} &=&
\me^{N/\lambda}\sum_{j=0}^{N-1}
\frac{(N/\lambda)^{j}\me^{-N/\lambda}}{(N-j)j!} \nonumber \\  &\approx&
\frac{\me^{N/\lambda}}{N(1-1/\lambda)} \sum_{j=0}^{N-1}
\frac{\me^{-N/\lambda}(N/\lambda)^{j}}{j!} \approx
\frac{\me^{N/\lambda}}{N(1-1/\lambda)} .
\end{eqnarray}

The second approximation we use is Stirling's formula implying that
$(N-1)!\sim \sqrt{2\pi (N-1)}(N-1)^{(N-1)}\me^{-(N-1)}$. Combining
these two approximations yields
\begin{eqnarray} \label{eq:SISdur}
\mE\left[ T^{(N)}\right] &\sim &\left(\frac{\lambda}{N}\right)^{N-1}\frac{
\sqrt{2\pi (N-1)}(N-1)^{(N-1)}}{\me^{(N-1)}N(1-1/\lambda)
}e^{N/\lambda} \nonumber \\ &=& \frac{ \sqrt{2\pi(N-1)}}{(\lambda
-1)N}\left(\frac{N-1}{N}\right)^{N-1}\frac{\lambda^N e^{N/\lambda}}{
\me^{N-1}}\nonumber\\
&\sim& \frac{\sqrt{2\pi}}{\lambda -1}\frac{\exp(\{\log\lambda -
1+1/\lambda\}N)}{\sqrt{N} },
\end{eqnarray}
as required.

Let $ A_n^{(N)}$ denote the total amount of time that the SIS
epidemic, initiated with a single infective, spends with $n$
infectious individuals. Then, from~\eqref{eq:stat:6a} and~\eqref{eq:stat:5a},
\begin{equation}
\label{eq:sisAn}
\mE\left[A_n^{(N)}\right]=\frac{\pi^{(N)} (n)}{\pi^{(N)} (0)}=\frac{\pi^{(N)} (n)}{1-\pi^{(N)} (0)}\mE\left[T^{(N)}\right].
\end{equation}
Using~\eqref{eq:stat:8}, the first equation in~\eqref{eq:sisAn}
and~\eqref{eq:SIS:1}, the mean total number of infectives during the
course of a supercritical epidemic is
\begin{eqnarray}
\mE\left[C^{(N)}\right] = \sum_{n=1}^N n \mE\left[A_n^{(N)}\right] &=& \sum_{n=1}^N n \frac{\pi^{(N)} (n)}{\pi^{(N)} (0)}
\nonumber \\ &=& \sum_{n=1}^N \frac{(N-1)!}{(N-n)!} \left(
\frac{\lambda}{N}
\right)^{n-1} \label{eq:SIS:f0.5} \\
& \sim & \frac{\sqrt{2 \pi}}{\lambda} \sqrt{N} \exp \left( \{\log
\lambda - 1 + 1/\lambda \} N \right).\label{eq:SIS:f1}
\end{eqnarray}
The derivation of \eqref{eq:SIS:f1} is similar to but simpler than
that of $\mE\left[ T^{(N)}\right]$.

Note that the second equation in~\eqref{eq:sisAn} gives
\begin{equation*}
\mE\left[C^{(N)}\right]=\left(\sum_{n=1}^N \frac{n \pi^{(N)} (n)}{1-\pi^{(N)} (0)}\right)\mE\left[T^{(N)}\right].
\end{equation*}
The distribution $\tilde{\bmpi}^{(N)}=(\tilde{\pi}_1^{(N)},
\tilde{\pi}_2^{(N)},\ldots, \tilde{\pi}_N^{(N)})$,  where
$\tilde{\pi}_n^{(N)}=\pi^{(N)}(n)/(1-\pi^{(N)}(0))$, gives a
``quasi-equilibrium" distribution for the SIS epidemic.  Thus, the
mean total number of infectives in the epidemic is given by the mean
number of infectives in quasi-equilibrium multiplied by the mean
duration of the epidemic.  When the epidemic is supercritical
($\lambda>1$), the distribution of $\tilde{\bmpi}^{(N)}$ is concentrated
on values close to $N(1-\lambda^{-1})$, as indicated previously,
which explains the simple multiplicative relationship between the
approximations~\eqref{eq:SISdur} and~\eqref{eq:SIS:f1}.

\subsubsection{Mean extinction time from quasi-endemic equilibrium} \label{S:SIS2}

The above calculations of $\mE\left[ T^{(N)}\right]$ are insensitive to the
distribution of $Q$. However, for supercritical SIS epidemics there
is interest in the time to extinction of the epidemic starting from
the quasi-endemic equilibrium of $N (1-1/\lambda)$ infectives. We
outline how the mean time to extinction from the quasi-endemic
equilibrium, $\mE\left[T_Q^{(N)}\right]$, does depend upon the distribution of $Q$.
The epidemic initiated from a single infective either goes extinct
very quickly or takes-off and reaches an endemic equilibrium of a
proportion $(\lambda -1)/\lambda$ of the population infected, see
\cite{Kryscio}. The epidemic then spends a long time fluctuating
about the endemic equilibrium before eventually going extinct. This
can be seen from $\bmpi^{(N)}$, with most of the probability mass
centred about $(\lambda -1) N/\lambda$ infectives. There has been
considerable interest in investigating the distribution of the time
to extinction from the endemic equilibrium, see for example
\cite{Kryscio}, \cite{AndDje}, \cite{Nas99} and~\cite{BN10}. This is
a difficult problem on which to make analytical progress. In
\cite{AndDje}, it was shown that, in the limit as $N \rightarrow
\infty$, for $Q \sim {\rm Exp} (1)$ the time to extinction divided
by $\mE\left[T_Q^{(N)}\right]$ converges in distribution to an exponential random
variable with mean 1. Moreover, $\mE\left[T_Q^{(N)}\right]\sim \sqrt{2 \pi/N}
\lambda \exp (N \{ \log \lambda + 1/\lambda -1\}) /(\lambda-1)^2 =
\mu^{(N)}/ (1-1/\lambda)$, where $\mu^{(N)} = \mE\left[T^{(N)}\right]$. It is conjectured
that an exponential distribution for the time to extinction holds
more generally than for $Q \sim {\rm Exp} (1)$, but even computing
$\mE\left[T_Q^{(N)}\right]$ up to leading terms in $N$ has proved difficult. By
studying Gaussian approximations for the endemic equilibrium
qualitative results on the time to extinction have been obtained,
see \cite{Nas99} and \cite{BN10}. Whilst, such approaches have given
a qualitative understanding of extinction of SIS epidemics, the
estimates obtained for the mean time to extinction are incorrect by
orders of magnitude. Moreover, it is noted in \cite{N14} that
simulation results suggest that the distribution of $Q$ does affect
the mean time to extinction from the quasi-endemic equilibrium,
which is not predicted by using the qualitative Gaussian
approximation.

Let $p_Q$ denote the extinction probability of a branching process,
in which individuals have {\it iid} infectious periods according to
$Q$ and whilst alive, give birth at the points of a homogeneous
Poisson point process with rate $\lambda$. We show in Lemma
\ref{lem:equil} that  $\mE\left[T_Q^{(N)}\right]$ depends upon the distribution $Q$
through $p_Q$.
\begin{lem} \label{lem:equil}
For $\lambda >1$ and $var (Q)<\infty$,
\[ \mE\left[T_Q^{(N)}\right] \sim \frac{1}{1-p_Q} \mu^{(N)}. \]
\end{lem}

We give a formal proof of Lemma~\ref{lem:equil} in Appendix B, where we also show that
the mean extinction time of the SIS epidemic $\sim \frac{1}{1-p_Q} \mu^{(N)}$ whenever it starts with a strictly positive fraction of the population infected. Here we present a heuristic proof of Lemma \ref{lem:equil}.
The requirement that $var(Q)< \infty$ is almost
certainly not necessary but is assumed in the proof below. As noted
above, the supercritical SIS epidemic will either quickly go extinct
or will take-off and reach the endemic equilibrium. Let $1-P_Q^{(N)}$
denote the probability that the total number of infectives in the
epidemic equals ${\left \lfloor{\epsilon N}\right \rfloor}$ at some point in time, for some $0 <
\epsilon < (\lambda -1)/\lambda$. (Throughout the paper, $\left \lfloor{x}\right \rfloor$ denotes the greatest integer $\le x$ and $\left \lceil{x}\right \rceil$ denotes the smallest integer $\ge x$.)  Then it is straightforward to
show, using a branching process approximation (see, for example,
\cite{Whittle} and~\cite{BallDonnelly}), that $P_Q^{(N)} \rightarrow
p_Q$ as $N \rightarrow \infty$. Then the mean duration of an
epidemic, initiated from a single infective satisfies
\begin{eqnarray} \label{eq:pi:5}
\mu^{(N)} = \frac{1}{\pi^{(N)} (1)} &=& P_Q^{(N)} A_Q^{(N)} +
(1-P_Q^{(N)}) \left\{ B_Q^{(N)} + \mE\left[T_Q^{(N)}\right]
\right\},
\end{eqnarray}
where $A_Q^{(N)}$ is the mean duration of an epidemic which never
reaches ${\left \lfloor{\epsilon N}\right \rfloor}$ infected (epidemic dies off quickly) and
$B_Q^{(N)}$ is the mean time take for the epidemic to reach the endemic
equilibrium given it reaches ${\left \lfloor{\epsilon N}\right \rfloor}$ infected. The definition of $B_Q^{(N)}$ is imprecise and
correspondingly we take $B^{(N)}_Q$ to be the mean time to reach
$[(\lambda -1) N/\lambda]$ infectives given that the epidemic takes
off. For the case $Q \sim {\rm Exp} (1)$, it is shown in
\cite{AndDje} that $A_Q^{(N)} = O(1)$ and $B_Q^{(N)}= O (\log N)$. Therefore
assuming that for general $Q$, $A_Q^{(N)}, B_Q^{(N)} = o (\mu^{(N)})$, we have
that
\begin{eqnarray} \label{eq:pi:6}
\mE\left[T_Q^{(N)}\right] & = & \frac{1}{1- P_Q^{(N)}} \left\{ \frac{1}{\pi^{(N)} (1)} - P_Q^{(N)}
A_Q^{(N)} - (1-P_Q^{(N)})  B_Q^{(N)} \right\} \nonumber \\
& \approx & \frac{1}{1-p_Q} \times \frac{1}{\pi^{(N)}(1)} =
\frac{1}{1-p_Q} \mu^{(N)},
\end{eqnarray}
Therefore we can immediately see the role of $Q$ in $\mE\left[T_Q^{(N)}\right]$.
Specifically, the
greater the extinction probability $p_Q$, the longer the epidemic
will on average persist, given that it takes off and becomes established.  Note that, subject to $\mE[Q]=1$, the extinction
probability is least when $Q$ is constant ({\it i.e.}~$\mP(Q=1)=1$),
so the model with a constant infectious period has the shortest mean
time to extinction starting from quasi-endemic equilibrium.

\subsection{Household SIS epidemic} \label{S:house}

The final special case we consider is the household SIS epidemic
model. Consider a fixed community consisting of $m$ households
which, for simplicity of exposition, all have the same size $h$, so
the population size is  $N_h=mh$. Our results extend
straightforwardly to the case where the household sizes are unequal.
We are particularly interested in the case where $m$, and hence
$N_h$, is large. Infectious individuals have {\it iid} infectious
periods according to $Q$, after which they become susceptible again.
While infectious an individual makes two types of contacts: the
individual makes global infectious contacts at rate $\lambda_G$,
each time the contacted person is selected independently and
uniformly at random from the whole community, including individuals
in the same household, and the individual makes local infectious
contacts at rate $\lambda_L$ with any given individual in their
household. Therefore, an individual makes infectious contacts at a
total rate of $\lambda_G + (h-1) \lambda_L$. By examining the
within-household dynamics of the SIS epidemic in the initial stages
of the epidemic and at the quasi-endemic equilibrium, we obtain
interesting, and perhaps unexpected, insensitivity results for the
household SIS epidemic model.

For large $m$, the initial stages of the household SIS epidemic can
be approximated by a branching process; see \cite{Ball99}, where the
approximation is made fully rigorous using a coupling argument. The
branching process approximation is similar to that used for the
household SIR epidemic, \cite{BMST97}, with individuals in the
approximating branching process corresponding to within-household
epidemic outbreaks in the epidemic. For large $m$, in the initial
stages of the household SIS epidemic the probability that a global
infectious contact is with an infectious household (a household
containing at least one infective) is very small. Therefore, we
assume that all global infectious contacts are with totally
susceptible households and we consider the epidemic within a
household, ignoring for the moment global infectious contacts,
initiated by a single infective and without any additional global
infections from outside.

Let $S$ denote the total severity of such a
within-household epidemic, where the severity is the sum of the infectious
periods of all infectives during the course of the epidemic from the
initial infective until the epidemic within the household ceases.
Then, conditional upon $S$, the total number of global infectious contacts emanating from
the household has a Poisson distribution with mean $\lambda_G S$, so the
 basic reproduction number of the approximating
branching process is
$R_\ast = \lambda_G \mE[S]$. The household SIS epidemic is said to be
subcritical, critical or supercritical if $R_\ast <1$, $R_\ast =1$
or $R_\ast >1$, respectively. The above expression for $R_\ast$
holds also for the household SIR epidemic, where it is known that $\mE[S]$,
the mean severity of the within-household epidemic, depends upon the
distribution of $Q$, as does both the size of a major outbreak and the distribution of the ultimate number of susceptibles in a typical household in the event of a major outbreak.  The following lemma shows that all of the corresponding quantities for the households SIS epidemic are insensitive to the distribution of $Q$.

\begin{lem} \label{lem:house}
For any $Q$, satisfying $\mE[Q]=1$,
\begin{eqnarray} \label{eq:house:2} R_\ast &=& \lambda_G \sum_{n=1}^h \frac{(h-1)!}{(h-n)!}
\lambda_L^{n-1},
\end{eqnarray} ({\it cf.}~\cite{Ball99}, equation (8)).

For $R_\ast >1$, in the limit as $m \rightarrow \infty$, there
exists an endemic equilibrium with a proportion $z$ of the
population infected, where $z$ is the unique non-zero solution of
\begin{eqnarray} \label{eq:house:3} s = \sum_{i=0}^h i
\phi_i (s),
\end{eqnarray}
with
\begin{eqnarray} \label{eq:house:4} \phi_0 (s) = \left\{ 1 + \sum_{k=1}^h \prod_{i=1}^k \frac{
(h+1 -i) (\lambda_G s + (i-1) \lambda_L)}{i} \right\}^{-1}
\end{eqnarray} and, for $1 \leq k \leq h$,
\begin{eqnarray} \label{eq:house:5} \phi_k (s) = \frac{(h+1 -k) (\lambda_G s + (k-1) \lambda_L) }{k} \phi_{k-1} (s).
\end{eqnarray}
Further, in the limit as $m \rightarrow \infty$, at the endemic equilibrium, the distribution of the number of infectives in a typical household is given by $(\phi_0(z), \phi_1(z),\dots,\phi_h(z))$.
\end{lem}
\begin{proof}
The within-household epidemic without additional global infections
is simply a homogeneously mixing SIS epidemic with $N=h$ and
$\lambda/N = \lambda_L$, so the insensitivity results for the
homogeneously mixing SIS epidemic are applicable. Of primary
interest, this means that
\begin{eqnarray} \label{eq:house:1} \mE[ S] &=&\sum_{n=1}^h \frac{(h-1)!}{(h-n)!}
\lambda_L^{n-1},
\end{eqnarray}
(recall~\eqref{eq:stat8a} and~\eqref{eq:SIS:f0.5}), whence
\eqref{eq:house:2} follows, regardless of the distribution of $Q$.
It should be noted that the distribution of $S$ does depend upon the
distribution of $Q$, so the probability that the epidemic takes off,
corresponding to the approximating branching process not going
extinct, does depend upon the distribution of $Q$.

We turn our attention to the quasi-endemic equilibrium in the case
$R_\ast > 1$. For the Markov case, it is shown in \cite{N06},
Section 4, that, in the limit as $m \rightarrow \infty$, there
exists a stable endemic equilibrium satisfying~\eqref{eq:house:3} to~\eqref{eq:house:5}. We again consider a
within-household epidemic but now with a constant global force of
infection from outside the household. Given that a proportion $z$ of
the population is infected, each individual in a household receives
global infectious contacts at the points of a homogeneous Poisson
point process with rate $z \lambda_G$. Therefore letting $\alpha (n)
= (h-n) \{ z \lambda_G + n \lambda_L \}$, new infections take place
within the household at rate $\alpha (n)$ if there are currently $n$
infectives and $h-n$ susceptibles. Again individuals within the
household have {\it iid} infectious
periods distributed according to $Q$. The within-household epidemic
satisfies the generic framework of Section \ref{S:generic} and it
follows from~\eqref{eq:house:5} that $\bmphi$ satisfies the detailed
balance equation~\eqref{eq:zachary:1}.  Thus $\bmphi$ is the
stationary distribution of the within-household epidemic, regardless
of the distribution of $Q$, so at the endemic equilibrium of the
household SIS epidemic, both the proportion of the population
infected and the distribution of the number infected in a typical
household are insensitive to the distribution of $Q$.
\end{proof}

\appendix

\section{Proof of Lemma~\ref{lem:sis}}
First note from~\eqref{eq:SIS:2} that
\begin{equation}
\label{ETNbirthday}
\mE\left[ T^{(N)}\right] = \sum_{n=1}^N \prod_{i=1}^{n-1} \left(1 - \frac{i}{N} \right)\frac{\lambda^{n-1}}{n}.
\end{equation}

Suppose that $\lambda<1$.  Then, for any $k \in \mathbb{N}$,
\[
\liminf_{N \to \infty} \mE\left[ T^{(N)}\right] \ge \liminf_{N \to \infty}\sum_{n=1}^k \prod_{i=1}^{n-1} \left(1 - \frac{i}{N} \right)\frac{\lambda^{n-1}}{n} = \sum_{n=1}^k \frac{\lambda^{n-1}}{n},
\]
and letting $k \to \infty$ yields
\begin{equation}
\label{liminfLB}
\liminf_{N \to \infty} \mE\left[ T^{(N)}\right] \ge - \frac{\log(1-\lambda)}{\lambda}.
\end{equation}
Also, for any $k \in \mathbb{N}$,
\[
\limsup_{N \to \infty}\mE\left[ T^{(N)}\right] \le \limsup_{N \to \infty}\sum_{n=1}^k \prod_{i=1}^{n-1} \left(1 - \frac{i}{N} \right)\frac{\lambda^{n-1}}{n}+\sum_{n=k+1}^{\infty} \lambda^{n-1}
=\sum_{n=1}^k \frac{\lambda^{n-1}}{n}+\frac{\lambda^k}{1-\lambda},
\]
and letting $k \to \infty$ yields
\begin{equation}
\label{limsupUB}
\limsup_{N \to \infty} \mE\left[ T^{(N)}\right] \le - \frac{\log(1-\lambda)}{\lambda}.
\end{equation}
Combining~\eqref{liminfLB} and~\eqref{limsupUB} yields~\eqref{eq:ls1}.

Suppose that $\lambda=1$.  Then, setting $\lambda=1$ in~\eqref{ETNbirthday} and noting that
\[
\prod_{i=1}^{n-1} \left(1 - \frac{i}{N} \right) \le \prod_{i=1}^{n-1} \exp\left(-\frac{i}{N}\right)=\exp\left(-\frac{n(n-1)}{2N}\right),
\]
yields that, for any $L>0$,
\[
\mE\left[ T^{(N)}\right] \le \sum_{n=1}^{\left \lceil{L\sqrt{N}}\right \rceil}\frac{1}{n}+
\sum_{\left \lceil{L\sqrt{N}}\right \rceil+1}^N \frac{1}{n}\exp\left(-\frac{n(n-1)}{2N}\right)
\le \sum_{n=1}^{{\left \lceil{L\sqrt{N}}\right \rceil}}\frac{1}{n}+
\exp\left(-L^2/2\right)\sum_{\left \lceil{L\sqrt{N}}\right \rceil+1}^N \frac{1}{n}.
\]
Hence,
\begin{equation}
\label{limsupUB1}
\limsup_{N \to \infty} \frac{\mE\left[ T^{(N)}\right]}{\frac{1}{2}\log N} \le 1+\exp(-L^2/2).
\end{equation}

Setting $\lambda=1$ in~\eqref{ETNbirthday} yields that, for any $K>0$,
\[
\mE\left[T^{(N)}\right] \ge \sum_{n=1}^{\left \lceil{K\sqrt{N}}\right \rceil}\frac{1}{n}
\prod_{i=1}^{n-1} \left(1 - \frac{i}{N} \right)\ge
\left(\sum_{n=1}^{\left \lceil{K\sqrt{N}}\right \rceil}\frac{1}{n}\right)\prod_{i=1}^{\left \lceil{K\sqrt{N}}\right \rceil}\left(1-\frac{i}{N} \right).
\]
Now~\eqref{birthday} implies that
\[
\lim_{N \to \infty}\prod_{i=1}^{\left \lceil{K\sqrt{N}}\right \rceil}\left(1-\frac{i}{N} \right)=\exp(-K^2/2),
\]
whence
\begin{equation}
\label{liminfLB1}
\liminf_{N \to \infty} \frac{\mE\left[ T^{(N)}\right]}{\frac{1}{2}\log N} \ge \exp(-K^2/2).
\end{equation}
Letting $L \to \infty$ in~\eqref{limsupUB1} and $K \downarrow 0$ in~\eqref{liminfLB1} yields~\eqref{eq:ls2}.

To prove~\eqref{eq:ls3} we show that
\begin{equation}
\label{limbounds3}
\lim_{N \to \infty} N \sum_{j=0}^{N-1}\frac{(N/\lambda)^{j}\me^{-N/\lambda}}{(N-j)j!}
=\frac{\lambda}{\lambda-1}.
\end{equation}
Note that this makes fully rigorous the approximation at~\eqref{eq:dur:1} and~\eqref{eq:ls3} then follows using~\eqref{eq:SIS:3} and Stirling's formula, as at~\eqref{eq:SISdur}.

Fix $\epsilon \in (0, \lambda^{-1})$ and let $A_1^\epsilon=\{j \in \mathbb{Z}:0 \le j <N(\lambda^{-1}-\epsilon)\}, A_2^\epsilon=\{j \in \mathbb{Z}:N(\lambda^{-1}-\epsilon) \le j \le N(\lambda^{-1}+\epsilon) \}$ and $A_3^\epsilon=\{j \in \mathbb{Z}:N(\lambda^{-1}+\epsilon) < j \le N-1\}$.
Further, let $X^{(N)}$ denote a Poisson random variable with
mean $N/\lambda$.  Then, using Chebyshev's inequality, $\mP\left(X^{(N)} \in A_1^\epsilon\right) \to 0$ and $\mP\left(X^{(N)} \in A_2^\epsilon\right) \to 1$ as $N \to \infty$.  Also, by large deviation theory, there exists $a>0$, independent of $N$, such that $\mP\left(X^{(N)}>N(\lambda^{-1}+\epsilon)\right) \le \me^{-aN}$.  Now
\begin{equation}
\label{Aeps1}
N\sum_{j \in A_1^\epsilon}\frac{(N/\lambda)^{j}\me^{-N/\lambda}}{(N-j)j!}
\le \frac{1}{\lambda^{-1}-\epsilon}\mP\left(X^{(N)} \in A_1^\epsilon\right)
\to 0 \quad \mbox{as} \quad N \to \infty
\end{equation}
and
\begin{equation}
\label{Aeps2}
N\sum_{j \in A_3^\epsilon}\frac{(N/\lambda)^{j}\me^{-N/\lambda}}{(N-j)j!}
\le N \mP\left(X^{(N)} > N(\lambda^{-1}+\epsilon)\right)
\to 0 \quad \mbox{as} \quad N \to \infty.
\end{equation}
Also,
\begin{equation*}
\frac{1}{1-\lambda^{-1}+\epsilon}\mP\left(X^{(N)} \in A_2^\epsilon\right)
\le N\sum_{j \in A_2^\epsilon}\frac{(N/\lambda)^{j}\me^{-N/\lambda}}{(N-j)j!}
\le \frac{1}{1-\lambda^{-1}-\epsilon}\mP\left(X^{(N)} \in A_2^\epsilon\right),
\end{equation*}
whence, using~\eqref{Aeps1}, \eqref{Aeps2} and $\lim_{N \to \infty}\mP\left(X^{(N)} \in A_2^\epsilon\right)=1$,
\begin{equation}
\label{limbounds1}
\liminf_{N \to \infty} N \sum_{j=0}^{N-1}\frac{(N/\lambda)^{j}\me^{-N/\lambda}}{(N-j)j!} \ge \frac{1}{1-\lambda^{-1}+\epsilon}
\end{equation}
and
\begin{equation}
\label{limbounds2}
\limsup_{N \to \infty} N \sum_{j=0}^{N-1}\frac{(N/\lambda)^{j}\me^{-N/\lambda}}{(N-j)j!} \le \frac{1}{1-\lambda^{-1}-\epsilon}.
\end{equation}
Letting $\epsilon \downarrow 0$ in~\eqref{limbounds1} and~\eqref{limbounds2} yields~\eqref{limbounds3}, as required.

\section{Proof of Lemma~\ref{lem:equil}}
For $i=0,1,\ldots,N$, let $T^{(N)}(i)=\inf\{t:Y^{(N)}_t=i\}$ be the first
time that the number of infectives $Y^{(N)}_t$ in the SIS epidemic equals
$i$, with the convention that $T^{(N)}(i)=\infty$ if $Y^{(N)}_t$ never
reaches $i$.  Further, let $F_N(i)=\{T^{(N)}(i)<\infty\}$ and
$F_N^\mc(i)=\{T^{(N)}(i)=\infty\}$.  Also, let $y^{(N)}=\left
\lfloor{(\lambda -1) N/\lambda} \right \rfloor$ and, for $k \in
(0,\infty)$, let $y_k^{(N)}=\left \lfloor{(\lambda -1)
N/\lambda-k\sqrt{N}}\right \rfloor$.  Then
\begin{equation*}
\mu^{(N)}=\mE\left[T^{(N)}\right]=\mE\left[T^{(N)} 1_{F_N^\mc(y_k^{(N)})}\right]+\mE\left[T^{(N)} 1_{F_N(y_k^{(N)})}\right].
\end{equation*}
Further,
\begin{equation*}
\mE\left[T^{(N)} 1_{F_N(y_k^{(N)})}\right]=\mE\left[T^{(N)}(y_k^{(N)}) 1_{F_N(y_k^{(N)})}\right]+\mE\left[\left(T^{(N)}-T^{(N)}(y_k^{(N)})\right) 1_{F_N(y_k^{(N)})}\right].
\end{equation*}
Thus,
\begin{equation}
\label{mun} \mu^{(N)}=a^{(N)}_k+b^{(N)}_k,
\end{equation}
where
\begin{equation}
\label{munan}
a^{(N)}_k=\mE\left[T^{(N)} 1_{F_N^\mc(y_k^{(N)})}\right]+\mE\left[T^{(N)}(y_k^{(N)}) 1_{F_N(y_k^{(N)})}\right]
\end{equation}
and
\begin{equation}
\label{munbn}
b^{(N)}_k=\mE\left[\left(T^{(N)}-T^{(N)}(y_k^{(N)})\right) 1_{F_N(y_k^{(N)})}\right].
\end{equation}

Recall that $A^{(N)}_n$ is the total time that the epidemic spends with $n$ individuals infective.  Note that if
$T^{(N)}(y_k^{(N)})=\infty$ then $T^{(N)} \le \sum_{n=1}^{y_k^{(N)}} A^{(N)}_n$, and if $T^{(N)}(y_k^{(N)})<\infty$ then $T^{(N)}(y_k^{(N)})\le \sum_{n=1}^{y_k^{(N)}} A^{(N)}_n$, so
\begin{equation}
\label{anbound}
0 \le a^{(N)}_k \le \sum_{n=1}^{y_k^{(N)}} \mE\left[A^{(N)}_n\right]=\sum_{n=1}^{y_k^{(N)}}\frac{(N-1)!}{n (N-n)!} \left( \frac{\lambda}{N} \right)^{n-1},
\end{equation}
using~\eqref{eq:stat:6a} and~\eqref{eq:SIS:1}.  Denote the right hand sum in~\eqref{anbound} by $c^{(N)}_k$.  Then rearranging as at~\eqref{eq:SIS:3} yields
\begin{equation}
\label{cn}
c^{(N)}_k=(N-1)!(\lambda/N)^{N-1}\me^{N/\lambda}\sum_{j=z_k^{(N)}}^{N-1}
\frac{(N/\lambda)^{j}\me^{-N/\lambda}}{(N-j)j!},
\end{equation}
where $z_k^{(N)}=\left \lceil{N\lambda^{-1}+k\sqrt{N}}\right \rceil$.  Omitting the details, fixing $\epsilon \in (0,\lambda^{-1})$, splitting the sum in~\eqref{cn} into $z_k^{(N)} \le j \le N(\lambda^{-1}+\epsilon)$ and $N(\lambda^{-1}+\epsilon)<j\le N-1$, invoking the central limit theorem and letting $\epsilon \downarrow 0$ yields
\begin{equation}
\label{cnlimit} \lim_{N \to \infty}
\frac{c^{(N)}_k}{\mu^{(N)}}=\Phi(-k\sqrt{\lambda}),
\end{equation}
where $\Phi$ denotes the standard normal cumulative distribution function.  Thus, recalling~\eqref{limbounds3}, for any $k>0$,
\begin{equation}
\label{anlimit} 0 \le \liminf_{N \to \infty} \frac{a^{(N)}_k}{\mu^{(N)}} \le
\limsup_{N \to \infty} \frac{a^{(N)}_k}{\mu^{(N)}} \le
\Phi(-k\sqrt{\lambda}).
\end{equation}
(Note that a similar argument using~\eqref{eq:bp:6} shows that the
mean time a supercritical branching process takes to reach size $n$,
given that it does not go extinct, is $O(\log n)$.)

We next determine $\lim_{N \to \infty} \mP\left(F_N(y^{(N)}_k)\right)$ $(k \in
(0,\infty))$.  For $\lambda>0$, let $\mathcal{B}_{\lambda}$ denote a
branching process, with one ancestor, in which individuals have {\it
iid} lifetimes according to $Q$ and, whilst alive, give birth at the
points of a homogeneous Poisson point process with rate $\lambda$.
Let $p_Q(\lambda)$ denote the probability that
$\mathcal{B}_{\lambda}$ does  go extinct and, for $t \ge 0$, let
$Y_t(\lambda)$ denote the number of individuals alive in
$\mathcal{B}_{\lambda}$ at time $t$.  Observe that by Lemma 4.1
of~\cite{BJM07}, $p_Q(\lambda') \to p_Q(\lambda)$ as $\lambda' \to
\lambda$. Note that $p_Q = p_Q (\lambda)$.

For $i=0,1,\ldots$, let $T_\lambda(i)=\inf\{t:Y_t(\lambda)=i\}$,
where $T_\lambda(i)=\infty$  if $Y_t(\lambda)$ never reaches $i$.
Fix $\delta \in [0,1)$.  Then, whilst $Y^{(N)}_t \le N\delta$, the SIS
epidemic is bounded below by the branching process
$\mathcal{B}_{(1-\delta)\lambda}$ ({\it c.f.}~\cite{Whittle}), so
\begin{equation}
\label{bplbound} 1-p_Q((1-\delta)\lambda) \le
\mP\left(T_{(1-\delta)\lambda}(\left \lfloor{\delta N}\right \rfloor) <
\infty\right) \le \mP\left(T^{(N)}(\left \lfloor{\delta N}\right \rfloor) <
\infty\right).
\end{equation}
Further, whilst $Y^{(N)}_t \le y_k^{(N)}$, the SIS epidemic is bounded below
by the  branching process $\mathcal{B}_{1+\lambda k
N^{-\frac{1}{2}}}$.  Using for example equation (5.63)
of~\cite{HJV05},
\begin{equation*}
p_Q\left(1+\lambda k N^{-\frac{1}{2}}\right) \le 1-\frac{2 \lambda
k}{N^{\frac{1}{2}}\mE [Q^2](1+\lambda k N^{-\frac{1}{2}})^2} \le
\exp\left(-\frac{2 \lambda k}{N^{\frac{1}{2}}\mE [Q^2](1+\lambda k
N^{-\frac{1}{2}})^2}\right).
\end{equation*}
Thus, for any $\delta \in (0, (\lambda-1)/\lambda)$ and any $k > 0$,
\begin{eqnarray}
\mP\left(T^{(N)}(y^{(N)}_k)=\infty\,|\,T^{(N)}(\left \lfloor{\delta N}\right \rfloor) < \infty\right) &\le& \left(1-\frac{2 \lambda k}{N^{\frac{1}{2}}\mE [Q^2](1+\lambda k N^{-\frac{1}{2}})^2}\right)^{\left \lfloor{\delta N}\right \rfloor} \label{pdeltakbound}\\
&\le& \exp\left(-\frac{2 \left \lfloor{\delta N}\right \rfloor \lambda k}{N^{\frac{1}{2}}\mE [Q^2](1+\lambda k N^{-\frac{1}{2}})^2}\right)\nonumber\\
&\to& 0\quad \mbox{as } N \to \infty, \nonumber
\end{eqnarray}
which on combining with~\eqref{bplbound} and letting $\delta \downarrow 0$ yields
\begin{equation*}
\label{liminfpFN} \liminf_{N \to \infty} \mP\left(F_N(y^{(N)}_k)\right) \ge
1-p_Q(\lambda).
\end{equation*}
(Note that the bound \eqref{pdeltakbound} may not hold if $Q$ does not have an exponential distribution, since the excess lives of the $\left \lfloor{\delta N}\right \rfloor$ individuals alive at time $T^{(N)}(\left \lfloor{\delta N}\right \rfloor)$ will not be distributed as $Q$.  However, it is clear that there exist $\epsilon_1, \epsilon_2 >0$ such that $\mP\left(Y^{(N)}(\epsilon_1) \ge \epsilon_2 \left \lfloor{\delta N}\right \rfloor \right) \to 1$ as $N \to \infty$, where $Y^{(N)}(\epsilon_1)$ denotes the number of individuals at time $T^{(N)}(\left \lfloor{\delta N}\right \rfloor)$ whose excess life is at least $\epsilon_1$.  Conditioning on the number of offspring of those $Y^{(N)}(\epsilon_1)$ individuals and arguing as above shows that $\mP\left(T^{(N)}(y^{(N)}_k)=\infty\,|\,T^{(N)}(\left \lfloor{\delta N}\right \rfloor) < \infty\right) \to 0$ as $N \to \infty$.
A similar comment applies elsewhere in the proof.)

The SIS epidemic is bounded above by $\mathcal{B}_\lambda$, so
\begin{equation*}
\label{limsuppFN} \limsup_{N \to \infty} \mP\left(F_N(y^{(N)}_k)\right) \le
\limsup_{N \to \infty} \mP\left(T_\lambda(y^{(N)}_k)<\infty\right)= 1-p_Q(\lambda),
\end{equation*}
whence, for any $k>0$,
\begin{equation}
\label{limpFNk} \lim_{N \to \infty} \mP\left(F_N(y^{(N)}_k)\right)=1-p_Q(\lambda).
\end{equation}

Suppose now that $k>l>0$.  Arguing as above shows that
\begin{equation}
\label{pnlk}
\mP\left(T^{(N)}(y_l^{(N)})<\infty\,|\,T^{(N)}(y_k^{(N)})<\infty\right) \to 1\quad\mbox{as}\quad N \to \infty,
\end{equation}
 whence
\begin{multline}
\label{bklNlim}
\frac{b^{(N)}_k}{\mu^{(N)}}=\frac{1}{\mu^{(N)}}\mE\left[\left(T^{(N)}-T^{(N)}(y_l^{(N)})\right) 1_{F_N(y_l^{(N)})}\right]\\
+\frac{1}{\mu^{(N)}}\mE\left[\left(T^{(N)}(y_l^{(N)})-T^{(N)}(y_k^{(N)})\right)
1_{F_N(y_l^{(N)})}\right] +o(1) \quad \mbox{as } N \to \infty.
\end{multline}
Let $\theta^{(N)}=1+\lambda l N^{-\frac{1}{2}}$.  Whilst $y_k^{(N)} \le
Y^{(N)}_t \le y_l^{(N)}$, the epidemic is bounded below by
$\mathcal{B}_{\theta^{(N)}}$ but now starting with $y_k^{(N)}$ individuals.
Hence, using~\eqref{pnlk},
\begin{multline}
\label{eptylykbound}
\frac{1}{\mu^{(N)}}\mE\left[\left(T^{(N)}(y_l^{(N)})-T^{(N)}(y_k^{(N)})\right)
1_{F_N(y_l^{(N)})}\right]\\ \le
\frac{1}{\mu^{(N)}}\mE\left[\left(T_{{\theta^{(N)}}}(y_l^{(N)})-T_{{\theta^{(N)}}}(y_k^{(N)})\right)
1_{F_N(k,l)}\right]+o(1) \quad \mbox{as } N \to \infty,
\end{multline}
where $F_N(k,l)=\{T_{\theta^{(N)}}(y_l^{(N)})< \infty\}$, assuming that $\mathcal{B}_{\theta^{(N)}}$ reaches $y_k^{(N)}$.

Consider the branching process  $\mathcal{B}_{\lambda}$.  For
$n=1,2,\ldots$, let  $A_\lambda(n)$ be the total time that
$\mathcal{B}_{\lambda}$ spends with $n$ individuals alive.  (Thus,
in the notation of Section~\ref{S:BP}, $A_\lambda(n)=A_n$.)  Suppose
that $\lambda>1$.  Then, since
\begin{equation*}
T_\lambda(n) 1_{\{T_\lambda(n)<\infty\}} \le \sum_{i=1}^n A_\lambda(i)1_{\{T_\lambda(n)<\infty\}}
\le \sum_{i=1}^n A_\lambda(i),
\end{equation*}
it follows from~\eqref{eq:bp:6} that
\begin{equation*}
\mE\left[T_\lambda(n) \,|\, T_\lambda(n)<\infty\right] \le \frac{1}{\lambda
(1-p_Q(\lambda))}\sum_{i=1}^n \frac{1}{i}.
\end{equation*}

Return to the branching process $\mathcal{B}_{\theta^{(N)}}$.  The
expected time for it to reach $y_l^{(N)}$, starting from $y_k^{(N)}$
individuals, given that it does so, is less than the expected time
for $\mathcal{B}_{\theta^{(N)}}$ to reach $y_l^{(N)}$, starting from one
individual, again conditional upon it doing so.  Thus, since $\theta^{(N)}>1$,
\begin{equation}
\label{ebplkbound}
\mE\left[\left(T_{{\theta^{(N)}}}(y_l^{(N)})-T_{{\theta^{(N)}}}(y_k^{(N)})\right)
1_{F_N(k,l)}\right] \le
\left(1-p_Q(\theta^{(N)})^{y_k^{(N)}}\right)\frac{1}{1-p_Q(\theta^{(N)})}\sum_{i=1}^{y_l^{(N)}}
\frac{1}{i}.
\end{equation}
Theorem 5.5 of~\cite{HJV05} yields $1-p_Q(\theta^{(N)})\sim 2\lambda l
N^{-\frac{1}{2}}/\mE[Q^2]$, whence, using~\eqref{ebplkbound},
$\mE\left[\left(T_{{\theta^{(N)}}}(y_l^{(N)})-T_{{\theta^{(N)}}}(y_k^{(N)})\right)
1_{F_N(k,l)}\right]=O(N^{\frac{1}{2}} \log N)$ as $N \to
\infty$.   It then follows from~\eqref{limpFNk} to~\eqref{eptylykbound}
that
\begin{multline}
\label{bklNlim1}
\frac{b^{(N)}_k}{\mu^{(N)}}=\frac{1}{\mu^{(N)}}(1-p_Q(\lambda))\mE\left[T^{(N)}-T^{(N)}(y_l^{(N)})\,|\,T^{(N)}(y_l^{(N)})<\infty\right]
+o(1) \quad \mbox{as } N \to \infty.
\end{multline}

Dividing~\eqref{mun} by $\mu^{(N)}$ and letting $N \to \infty$ yields,
after using~\eqref{anlimit},
\begin{multline*}
1-\Phi(-k\sqrt{\lambda}) \le \liminf_{N \to \infty}
(1-p_Q(\lambda))\frac{\mE\left[T^{(N)}-T^{(N)}(y_l^{(N)})\,|\,T^{(N)}(y_l^{(N)})<\infty\right] }{\mu^{(N)}}
\\\le \limsup_{N \to
\infty} (1-p_Q(\lambda))\frac{\mE\left[T^{(N)}-T^{(N)}(y_l^{(N)})\,|\,T^{(N)}(y_l^{(N)})<\infty\right]
}{\mu^{(N)}} \le 1,
\end{multline*}
whence, letting $k \to \infty$,
\begin{equation*}
\mE\left[T^{(N)}-T^{(N)}(y_l^{(N)})\,|\,T^{(N)}(y_l^{(N)})<\infty\right] \sim \frac{1}{1-p_Q(\lambda)}
\mu^{(N)}.
\end{equation*}

The above shows that, for any $k >0$, the mean time to extinction
from $y_k^{(N)}=\left \lfloor{(\lambda -1) N/\lambda-k\sqrt{N}}\right
\rfloor$, $\mu^{(N)}_k$ say, satisfies $\mu^{(N)}_k \sim \frac{1}{1-p_Q}
\mu^{(N)}$, since $p_Q (\lambda) = p_Q$. Suppose that we start at the
endemic level with $y^{(N)}=\left \lfloor{(\lambda -1) N/\lambda} \right
\rfloor$ infectives. Then the mean time to extinction from the
endemic level satisfies
\begin{equation}
\label{muQDk}
\mu_Q^{(N)} = D_{Q,k}^{(N)} + \mu_k^{(N)},
\end{equation}
where $D_{Q,k}^{(N)}$ denotes the mean time that the epidemic takes
to reach $y_k^{(N)}$ for the first time. We prove that
$D_{Q,k}^{(N)} =O(N)$ for $1 < \lambda \leq 2$ and that
$D_{Q,k}^{(N)} =O(N\log N)$ for $\lambda >2$. 

Turning first to the case when $1 < \lambda \leq 2$, consider a subcritical branching
process with immigration, $\mathcal{B}^{(N)}$, where individuals
immigrate into the population at the points of a homogeneous Poisson
process with rate $\lambda N(1-1/\lambda)^2$. The lifetimes of
individuals in $\mathcal{B}^{(N)}$ are independent and identically
distributed according to $Q$, and whilst alive, individuals give
birth at the points of a homogeneous Poisson process with rate
$2-\lambda$. Let $B^{(N)}_t$ denote the total number of individuals
alive in $\mathcal{B}^{(N)}$ at time $t$. Then for any $t \geq 0$,
\cite{N14}, Corollary 2.1, gives that
\begin{equation} \label{ttm1} \sqrt{N} (B^{(N)}_t/N - (1-1/\lambda))
\convd N(0,1/\lambda) \quad \mbox{as } N \to \infty.
\end{equation}
It is shown in \cite{N14}, Section 3, that when $1 < \lambda \leq 2$, the SIS epidemic can be coupled to $\mathcal{B}^{(N)}$, such that for
all $t \geq 0$, $Y^{(N)}_t \leq B^{(N)}_t$. Therefore $D_{Q,k}^{(N)}
\leq \tilde{D}_{Q,k}^{(N)}$, the mean time that $\mathcal{B}^{(N)}$
takes to reach $y_k^{(N)}$ for the first time starting with $\left
\lfloor{(\lambda -1) N}\right \rfloor$ individuals. Therefore we
focus on computing an upper bound for $\tilde{D}_{Q,k}^{(N)}$.

Let $T_1^{(N)}$ denote the total length of time it takes for the
family trees of all individuals alive at time 0 to go extinct which
is stochastically smaller than the sum of the extinction times of
the family trees from each of the $\left \lfloor{(\lambda -1)
N}\right \rfloor$ individuals. Thus, using~\eqref{meansubcritext},
\begin{equation}
\label{tt0} \mE\left[T^{(N)}_1 \right] \le \frac{\lambda-1}{\lambda}
N\left\{ \frac{-\log(1-(2-\lambda))}{2-\lambda} \right\}=O(N) \quad
\mbox{as } N \to \infty.
\end{equation}
Let $\hat{B}^{(N)}_{1,t}$ denote the total progeny still alive at time $t$ of the family
trees of individuals who immigrate into the population in the
interval $[0,T_1^{(N)}]$. Then $\hat{B}^{(N)}_{1,t}$ is
independent of the family trees of individuals who immigrate into
the population prior to time 0 and $B^{(N)}_{T_1^{(N)}}
=\hat{B}^{(N)}_{1,T_1^{(N)}}$, where regardless of the value of
$T_1^{(N)}$, $\hat{B}^{(N)}_{1,T_1^{(N)}}$ is stochastically smaller than
$B^{(N)}_\ast$, the total number of individuals alive at time 0 of an
independent copy of $\mathcal{B}^{(N)}$ started at time $-\infty$.
Using \eqref{ttm1} it is straightforward to show that
\begin{eqnarray}
\label{tt1} \mP (\hat{B}^{(N)}_{1,T_1^{(N)}} <y_k^{(N)}) &\geq& \mP
(B^{(N)}_\ast <y_k^{(N)}) \nonumber \\
& \rightarrow & \Phi (-k \sqrt{\lambda})  \quad \mbox{as } N \to
\infty.
\end{eqnarray}
If $\hat{B}^{(N)}_{1,T_1^{(N)}} <y_k^{(N)}$, we know that the first
time $B^{(N)}_t$ reaches $y_k^{(N)}$ is less than or equal to
$T_1^{(N)}$. Otherwise, we consider $B^{(N)}_t$ at a sequence of
times $S_k^{(N)} = \sum_{i=1}^k T_i^{(N)}$, where for
$i=2,3,\ldots$, $T_i^{(N)}$ denotes the total length of time (from
$S_{i-1}^{(N)}$) it takes for the family trees of all individuals
alive at time $S_{i-1}^{(N)}$ to go extinct. This is stochastically
smaller than the sum of the extinction times of the family trees
from each of the $\hat{B}^{(N)}_{i-1,S_{i-1}^{(N)}}$ individuals
alive at time $S_{i-1}^{(N)}$, where $\hat{B}^{(N)}_{i-1,t}$ denotes
the total progeny still alive at time $t$ of the family trees of individuals who immigrate
into the population in the interval $[S_{i-2}^{(N)},S_{i-1}^{(N)}]$. $(S_0^{(N)} = T_0^{(N)} =0.)$ Note that the
$\hat{B}^{(N)}_{i,S_i^{(N)}}$ $(i=1,2,\dots,T_{i-1}^{(N)})$ are conditionally independent
given $\{T_i^{(N)}\}$ and, moreover, $T_i^{(N)}$ depends upon
$\{(\hat{B}^{(N)}_{1,S_1^{(N)}}, T_1^{(N)}), \ldots,
(\hat{B}^{(N)}_{i-1,S_{i-1}^{(N)}}, T_{i-1}^{(N)}) \}$ only through
$\hat{B}^{(N)}_{i-1,S_{i-1}^{(N)}}$. It is then straightforward
using similar arguments to \eqref{tt0} and \eqref{tt1} to show that
\begin{eqnarray}
\label{tt2} \mE\left[T^{(N)}_i | \hat{B}^{(N)}_{i-1,S_{i-1}^{(N)}} >
y_k^{(N)}\right] &=& \mE\left[\mE\left[T^{(N)}_i
|\hat{B}^{(N)}_{i-1,S_{i-1}^{(N)}} \right] \right] \nonumber \\
&\le&  \mE \left[\hat{B}^{(N)}_{i-1,S_{i-1}^{(N)}} |
\hat{B}^{(N)}_{i-1,S_{i-1}^{(N)}} > y_k^{(N)} \right]
\frac{-\log(1-(2-\lambda))}{2-\lambda} \nonumber\\ &=& O(N) \quad
\mbox{as } N \to \infty,
\end{eqnarray}
$B^{(N)}_{S_{i-1}^{(N)}} =\hat{B}^{(N)}_{i-1,S_{i-1}^{(N)}}$ and
\begin{eqnarray}
\label{tt3} \mP \left(\hat{B}^{(N)}_{i-1,S_{i-1}^{(N)}} <y_k^{(N)}
\right) &\geq& \mP
(B^{(N)}_\ast <y_k^{(N)}) \nonumber \\
& \rightarrow & \Phi (-k \sqrt{\lambda})  \quad \mbox{as } N \to
\infty.
\end{eqnarray}
Let $L^{(N)} = \min \{ l; \hat{B}^{(N)}_{l,S_l^{(N)}} < y_k^{(N)}
\}$.  Note that by \eqref{tt0} and \eqref{tt2}, there exists $M
<\infty$ such that for all $i=1,2,\ldots$ and for all sufficiently
large $N$, $\mE [ T_i^{(N)} | \hat{B}^{(N)}_{i-1,S_{i-1}^{(N)}} >
y_k^{(N)}] \leq MN$. Therefore
\begin{eqnarray}
\label{tt4} \mE [S_{L^{(N)}}^{(N)}] &=& \mE \left[ \mE \left[ \left.
\sum_{i=1}^{L^{(N)}} T_i^{(N)} \right|L^{(N)} \right] \right] \nonumber \\
&=& \mE \left[ \sum_{i=1}^{L^{(N)}}  \mE \left[ T_i^{(N)} |
\hat{B}^{(N)}_{i-1,S_{i-1}^{(N)}} > y_k^{(N)} \right] \right]
\nonumber \\
& \leq & M N \mE [L^{(N)}] = O (N) \quad \mbox{as } N \to \infty,
\end{eqnarray}
since for any $l \geq 1$ and 
$\limsup_{N \to \infty} \mP(L^{(N)}
> l) \leq (1 -(\Phi (-k \sqrt{\lambda}))^l$. 
Hence $\tilde{D}_{Q,k}^{N} \leq \mE [S_{L^{(N)}}^{(N)}] = O(N)$.


Before considering the case $\lambda>2$, it is fruitful to derive an upper bound for the 
expected time the SIS epidemic takes to reach its endemic level $y^{(N)}$ given that it 
is currently above that level.  Thus, suppose that the SIS epidemic starts with $k>y^{(N)}$ infectives and
define $T^{(N)}(y^{(N)})$ as before, but note that now
$\mP\left(T^{(N)}(y^{(N)})<\infty\right)=1$.  A simple calculation shows that, whilst
$Y_t^{(N)}>y^{(N)}$, $Y_t^{(N)}$ is bounded below by the subcritical branching
process $\mathcal{B}_{\eta^{(N)}}$, where $\eta^{(N)}=1-\lambda/N$, starting
from $k$ individuals.  It follows that $T^{(N)}(y^{(N)})$ is stochastically
smaller than the extinction time of this branching process, which in
turn is stochastically smaller than the sum of the extinction times
of the family trees from each of the $k$ initial individuals in
$\mathcal{B}_{\eta^{(N)}}$.  Thus, using~\eqref{meansubcritext}, for any $k > y^{(N)}$,
\begin{equation}
\label{mttoendabove}
\mE\left[T^{(N)}(y^{(N)})\,|\,Y_0=k\right] \le N\frac{-\log(\frac{\lambda}{N})}{1-\frac{\lambda}{N}}=O(N \log N) \quad \mbox{as } N \to \infty.
\end{equation}

Fix $\lambda>2$ and $k>0$.  For all $N>\left(k/\left(\frac{1}{2}-\frac{1}{\lambda}\right)\right)^2$ and all $n>y_k^{(N)}$, the infection rate satisfies
\begin{eqnarray*}
\frac{\lambda}{N}n(N-n) &\le& \frac{\lambda}{N}\left[\left(\frac{\lambda-1}{\lambda}\right)N-k\sqrt{N}\right]\left[N-\left(\frac{\lambda-1}{\lambda}\right)N+
k\sqrt{N}\right]\\
&=&\left(1-\frac{1}{\lambda}\right)N-2k\sqrt{N}-k^2 \lambda+k\lambda\sqrt{N}\\
&\le&\beta(N,k,\lambda),
\end{eqnarray*}
where $\beta(N,k,\lambda)=\left(1-\frac{1}{\lambda}\right)N+k\lambda\sqrt{N}$.
For such $N$ it follows that the SIS process $\{Y^{(N)}_t:t \ge 0\}$ can be coupled to a birth-death type process $\{\tilde{Y}^{(N)}_t:t \ge 0\}$ having birth rate given by $\alpha(n)=\beta(N,k,\lambda)$ if $y_k^{(N)}<n \le N$ and $\alpha(n)=0$ otherwise, such that if $Y^{(N)}_0=\tilde{Y}^{(N)}=y^{(N)}$ and these $y^{(N)}$ initial individuals have the same excess lifetimes in the two processes, then $Y^{(N)}_t \le \tilde{Y}^{(N)}_t$ for all $0 \le t \le T^{(N)}(y^{(N)}_k)$.

Let $T_1^{(N)}$ denote the total length of time it takes for the $y^{(N)}$ initial individuals in $\{Y^{(N)}_t:t \ge 0\}$ (or $\{\tilde{Y}^{(N)}_t:t \ge 0\}$) to all die.  Then, conditional upon $T_1^{(N)}$,
$\tilde{Y}^{(N)}_{T_1^{(N)}}$ has a Poisson distribution with mean bounded above by
\[
\beta(N,k,\lambda)\int_0^{T_1^{(N)}}{\rm P}(Q>u){\rm d}u \le \beta(N,k,\lambda)\int_0^{\infty}{\rm P}(Q>u){\rm d}u=\beta(N,k,\lambda),
\]
since ${\rm E}[Q]=1$.  Let $X^{(N)}$ denote a Poisson random variable with mean $\beta(N,k,\lambda)$.  Then,
the above coupling implies that
\[
{\rm P}\left(T^{(N)}(y^{(N)}_k)\le T_1^{(N)}\right) \ge {\rm P}\left(\tilde{Y}^{(N)}_{T_1^{(N)}}\le y^{(N)}_k\right) \ge
{\rm P}\left(X^{(N)}\le y^{(N)}_k \right).
\]
Straightforward application of the central limit theorem yields that
\[
{\rm P}\left(X^{(N)}\le y^{(N)}_k \right) \to \Phi\left(-k(\lambda+1)\sqrt{\frac{\lambda}{\lambda-1}}\right)\quad \mbox{as} \quad N \to \infty.
\]
Note also that $T_1^{(N)}$ is less than the sum of the excess lifetimes of the $y^{(N)}$ individuals alive at time $t=0$, so ${\rm E}\left[T_1^{(N)}\right]=O(N)$.

If $Y^{(N)}_{T_1^{(N)}}>y^{(N)}_k$, let $U_1^{(N)}=\min\{u \ge 0: Y^{(N)}_{T_1^{(N)}+u} \le y^{(N)}\}$.  Note that $U_1^{(N)}=0$ if $Y^{(N)}_{T_1^{(N)}}\le  y^{(N)}$ and~\eqref{mttoendabove} implies that ${\rm E}\left[U_1^{(N)}\,|\,Y^{(N)}_{T_1^{(N)}}> y^{(N)}\right]=O(N\log N)$.  Thus, ${\rm E}\left[U_1^{(N)}\right]=O(N \log N)$.  Now let $T_2^{(N)}$ be the total length of time it takes for the $Y^{(N)}_{T_1^{(N)}+U_1^{(N)}}$ individuals alive in $\{Y^{(N)}_t:t \ge 0\}$ at time $T_1^{(N)}+U_1^{(N)}$ to all die.  The argument now proceeds in a similar fashion to the case when $1<\lambda\le 2$ and, omitting the details, it is easily seen that if $\lambda>2$ then $\tilde{D}_{Q,k}^{N}=O(N\log N)$ and that this holds for all $\lambda >1$. It then follows using~\eqref{muQDk} that $\mu^{(N)}_Q \sim \frac{1}{1-p_Q} \mu^{(N)}$, as required.

Finally, for $i=1,2,\dots, N$, let $T_i^{(N)}=\min\{t>0: Y_t^{(N)}=0\}$ be the extinction time of the SIS epidemic given that initially there are $i$ infectives.  Then, together with Lemma~\ref{lem:equil},~\eqref{mttoendabove} implies that for any $\delta>0$,
\[
\mE\left[T_i^{(N)}\right] \sim \frac{1}{1-p_Q} \mu^{(N)} \quad \mbox{for any } i \ge N\delta.
\]

\section*{Acknowledgements}

We would like to thank Amaury Lambert for bringing \cite{Lambert11}
and the proof of \eqref{eq:bp:6} to our attention.  Tom Britton is grateful to the Swedish Research Council for financial support.

\end{document}